\documentclass[12pt,reqno]{amsart}

\usepackage{times}
\usepackage{graphicx}
\graphicspath{ {./images/} }

\usepackage{amsmath,amsthm,amsfonts,amscd,amssymb}
\usepackage{pst-node}
\usepackage{pst-plot}
\usepackage[all]{xy}
\usepackage{stackengine}

\usepackage{tikz, braids}
\usetikzlibrary{decorations.pathreplacing, arrows}

\newtheorem{theorem}{Theorem}[section]
\newtheorem{corollary}[theorem]{Corollary}

\newtheorem{lemma}[theorem]{Lemma}
\newtheorem{remark}[theorem]{Remark}
\newtheorem{proposition}[theorem]{Proposition}

\numberwithin{theorem}{section}
\numberwithin{equation}{section}

\newcommand{\norm}[1]{\left\Vert#1\right\Vert}

\newcommand{\la}{\langle}
\newcommand{\ra}{\rangle}

\newcommand{\Comp}{\mathbb{C}}

\newcommand{\g}{\mathbb{G}}

\newcommand{\M}{\mathcal{M}}
\newcommand{\n}{\mathbb{N}}
\newcommand{\re}{\mathbb{R}}

\newcommand{\z}{\mathbb{Z}}

\global\long\def\tp{\mathop{\xymatrix{*+<.7ex>[o][F-]{\scriptstyle \top}}
 } }

\begin{document}

\title[{\tiny A Khintchine inequality for central Fourier series on non-Kac CQGs}]{A Khintchine inequality for central Fourier series on non-Kac compact quantum groups}
% or if you want, simply \title{Title of the article}

\author[Sang-Gyun Youn]{Sang-Gyun Youn}
\address{Sang-Gyun Youn, 
Department of Mathematics Education, Seoul National University, 
Gwanak-Ro 1, Gwanak-Gu, Seoul 08826, Republic of Korea}
\email{s.youn@snu.ac.kr }

\maketitle

\begin{abstract}
The study of Khintchin inequalities has a long history in abstract harmonic analysis. While there is almost no possibility of non-trivial Khintchine inequality for central Fourier series on compact connected semisimple Lie groups, we demonstrate a strong contrast within the framework of compact quantum groups. Specifically, we establish a Khintchine inequality with operator coefficients for arbitrary central Fourier series in a large class of non-Kac compact quantum groups. The main examples include the Drinfeld-Jimbo $q$-deformations $G_q$, the free orthogonal quantum groups $O_F^+$, and the quantum automorphism group $\g_{aut}(B,\psi)$ with a $\delta$-form $\psi$.
\end{abstract}

\section{Introduction}

The Khintchine inequality has a long tradition in harmonic analysis, and its natural analogue has been studied extensively in noncommutative harmonic analysis \cite{Lu86,LuPi91,Pi12}. A noncommutative Khintchine inequality with operator coefficients for the cases $2\leq p<\infty$ is given by \small
\begin{equation}\label{eq-Khintchine-original}
 \norm{\sum_{i=1}^m \epsilon_i\otimes A_i}_{L^p(\Omega;S^p_n)}\lesssim \max\left\{ \norm{\left ( \sum_{i=1}^m A_i^*A_i \right )^{\frac{1}{2}}}_{S^p_n}, \norm{\left ( \sum_{i=1}^m A_iA_i^* \right )^{\frac{1}{2}}}_{S^p_n} \right\}
\end{equation}
\normalsize where $(\epsilon_i)_{i=1}^m$ is a set of independent and identically distributed Rademacher variables with $\epsilon_i:\Omega\rightarrow \left\{\pm 1\right\}$, and $S^p_n$ is the matrix algebra $M_n$ with the Schatten $p$-norm. When comparing two norms $\norm{\cdot}$ and $\norm{\cdot}'$, let us denote by $\norm{f}\lesssim \norm{f}'$ if there exists a constant $C$ such that
\begin{equation}
    \norm{f}\leq C\norm{f}'
\end{equation}
where $C$ is independent of the choice of $f$. Additionally, we write $\norm{\cdot}\approx \norm{\cdot}'$ if both $\norm{\cdot}\lesssim \norm{\cdot}'$ and $ \norm{\cdot}'\lesssim \norm{\cdot}$ hold.

Without using independent random variables, one of the most successful approaches has been the study of {\it lacunary sets}. This subject has a long and rich history for compact abelian groups \cite{GrHa13,Ka85,LoRo75,MaPi81}. Let us denote by $\text{Irr}(G)$ the set of all irreducible unitary representations $\pi:G\rightarrow \mathcal{U}(n_{\pi})$ of a compact group $G$. In the category of compact abelian groups, all irreducible representations are one-dimensional, i.e. $n_{\pi}\equiv 1$, and $\left\{\pi\right\}_{\pi\in \text{Irr}(G)}$ forms an orthonormal basis of $L^2(G)$. Moreover, there exists an infinite subset $E\subseteq \text{Irr}(G)$ such that \small
\begin{equation}
    \norm{\sum_{\pi \in E} \pi \otimes A_{\pi}}_{L^p(G;S^p_n)}\lesssim \max\left\{ \norm{\left ( \sum_{\pi\in E} A_{\pi}^*A_{\pi} \right )^{\frac{1}{2}}}_{S^p_n}, \norm{\left ( \sum_{\pi\in E} A_{\pi}A_{\pi}^* \right )^{\frac{1}{2}}}_{S^p_n} \right\}
\end{equation}
\normalsize by \cite{LoRo75} and \cite{Har99}.

There have been two general frameworks to study such infinite subsets $E\subseteq \text{Irr}(G)$ for non-abelian $G$, as mentioned in \cite{Pi20} for Sidon sets.

(1) One approach is to replace the compact abelian groups $G$ with the duals $\widehat{\Gamma}$ of discrete groups. This approach has led to various astounding positive results. For example, {\it freely independent} noncommutative random variables can be used to obtain Khintchine inequalities \cite{Bu99,HaPi93,PaPi05,Jun05,RiXu06}. For example, let $2\leq p\leq \infty$ and $E=\left\{g_i\right\}_{i=1}^m$ be a set of canonical generators of $\mathbb{F}_m$. Then\small
\begin{equation}\label{eq-Haagerup}
 \norm{\sum_{x\in E} \lambda_{x} \otimes A_x}_{L^p(\widehat{\mathbb{F}_m}; S^p_n)}\lesssim \max\left\{ \norm{\left ( \sum_{x\in E} A_x^*A_x \right )^{\frac{1}{2}}}_{S^p_n}, \norm{\left ( \sum_{x\in E} A_xA_x^* \right )^{\frac{1}{2}}}_{S^p_n} \right\}.
\end{equation}
\normalsize Here, $L^p(\widehat{\mathbb{F}_m}; S^p_n)$ is the noncommutative $L^p$-space from the tensor product von Neumann algebra $M_n \overline{\otimes}\mathcal{L}(\mathbb{F}_m)$ with the noncommutative measure $\text{tr}_n\otimes \tau$, where $\text{tr}_n$ is the usual trace on $M_n$ and $\tau= \la \cdot \delta_e,\delta_e\ra_{\ell^2(\mathbb{F}_m)}$ is the vacuum state on the free group von Neumann algebra $\mathcal{L}(\mathbb{F}_m)$. 

 (2) The other approach is to replace compact abelian groups with compact non-abelian groups $G$. In this case, the one-dimensional representations $\pi$ are replaced by matrix coefficient functions $\sqrt{n_{\pi}}\pi_{ij}$. However, this line of research encountered some disappointing discoveries for compact non-abelian Lie groups. For example, it was proven in \cite{GT80} that compact connected semisimple Lie groups do not admit an infinite $\Lambda(p)$-set, even in the scalar-valued case. Specifically, for any $1<p<\infty$, there is no infinite subset $E\subseteq \text{Irr}(G)$ such that
 \begin{equation}
 \norm{\sum_{\pi\in E}\sum_{i,j=1}^{n_{\pi}}a^{\pi}_{ij} \sqrt{n_{\pi}} \pi_{ij}}_{L^p(G)}\lesssim \left ( \sum_{\pi\in E}\sum_{i,j=1}^{n_{\pi}} |a^{\pi}_{ij}|^2 \right )^{\frac{1}{2}}.
 \end{equation}
One might hope to explore new possibilities by focusing on the {\it characters} $\chi_{\pi}=\sum_{i=1}^{n_{\pi}}\pi_{ii}$. However, it was shown in \cite{GST82} that for any $3\leq p<\infty$ and for any compact connected semisimple Lie groups $G$, there is no infinite subset $E\subseteq \text{Irr}(G)$ satisfying \small
\begin{equation}
    \norm{\sum_{\pi\in E} a_{\pi}\chi_{\pi} }_{L^p(G)}\lesssim \left ( \sum_{\pi\in E} |a_{\pi}|^2 \right )^{\frac{1}{2}}.
\end{equation}
\normalsize  

Despite the aforementioned negative conclusions in the category of compact non-abelian Lie groups, a positive result was found for a {\it twisted quantum group $\g=SU_q(2)$} with $0<q<1$. Specifically, \cite{Wa17} constructed an infinite subset $E\subseteq \text{Irr}(SU_q(2))$ such that
\begin{equation}
    \norm{\sum_{\pi\in E} a_{\pi}\chi_{\pi} }_{L^4(SU_q(2))}\lesssim \left ( \sum_{\pi\in E} |a_{\pi}|^2 \right )^{\frac{1}{2}}.
\end{equation}
This finding suggests potential for further exploration of the Khintchine inequality using characters on the {\it Drinfeld-Jimbo $q$-deformations} $\g=G_q$, and more generally on {\it non-Kac compact quantum groups}.

The main aim of this paper is to establish a very positive result in a large class of non-Kac compact quantum groups. Let $L^p(\g;S^p_n)$ denote the vector-valued noncommutative $L^p$-space $S^p_n[L^p(\g)]$ \cite{Pi98b,Pi03}. In Section 3, we prove the following Khintcine inequality with operator coefficients on general compact quantum groups (Theorem \ref{thm-main}): 
\begin{equation}\label{eq01}
   \norm{\sum_{\alpha \in \text{Irr}(\g)} \chi_{\alpha} \otimes  A_{\alpha}}_{L^p(\g; S^p_n)} \leq K_p\norm{\left (\sum_{\alpha\in \text{Irr}(\g)} A_{\alpha}^*A_{\alpha}\right )^{\frac{1}{2}}}_{S^p_n}
\end{equation}
for all $p=2^k$ with natural numbers $k$, where the constant $K_p$ is given by
\begin{equation}\label{eq02}
    K_p=\left ( \sum_{\alpha\in \text{Irr}(\g)} \norm{\chi_{\alpha}}_{L^{\infty}(\g)}^{2-\frac{4}{p}}  \left ( \frac{n_{\alpha}}{d_{\alpha}} \right )^{\frac{2}{p}} \right )^{\frac{1}{2}}\in [1,\infty].
\end{equation} 
Note that the whole set $E=\text{Irr}(\g)$ is taken to cover {\it all central Fourier series} $f\sim \sum_{\alpha\in \text{Irr}(\g)}\chi_{\alpha}\otimes A_{\alpha}$ with operator coefficients.

In Section 4, we prove that $K_p<\infty$ for all $2<p<\infty$ in a wide range of non-Kac compact quantum groups, including the Drinfeld-Jimbo $q$-deformations $G_q$, non-Kac free orthogonal quantum groups $O_F^+$ and non-Kac quantum automorphism groups $\g_{aut}(B,\psi)$ with a $\delta$-form $\psi$. In these cases, for any $1\leq r,s<\infty$, we obtain 
\begin{equation}
\norm{f}_{L^r(\g)}\approx \norm{f}_{L^s(\g)} 
\end{equation}
for arbitrary central Fourier series $f\sim \sum_{\alpha\in \text{Irr}(\g)}a_{\alpha}\chi_{\alpha}$ with scalar coefficients.

\section{Preliminaries}	

\subsection{Compact quantum groups and representations}

Within the von Neumann algebraic framework \cite{KuVa00,Va01,KuVa03}, a {\it compact quantum group} $\g$ is given by a triple $(L^{\infty}(\g),\Delta,h)$ where 
\begin{enumerate}
    \item $L^{\infty}(\g)$ is a von Neumann algebra, 
    \item $\Delta:L^{\infty}(\g)\rightarrow L^{\infty}(\g)\overline{\otimes}L^{\infty}(\g)$ is a normal unital $*$-homomorphism satisfying
\begin{equation}
\displaystyle (\Delta\otimes \text{id})\Delta= (\text{id}\otimes \Delta)\Delta,
\end{equation}
\item $h:L^{\infty}(\g)\rightarrow \Comp$ is a normal faithful state satisfying
\begin{equation}
(\text{id}\otimes h)(\Delta(a))=h(a)1=(h\otimes \text{id})(\Delta(a))
\end{equation}
for all $a\in L^{\infty}(\g)$. We call $h$ the {\it Haar state} on $\g$.
\end{enumerate}

A (finite dimensional) {\it unitary representation} of $\g$ is given by a unitary $v=\displaystyle \sum_{i,j=1}^{n_{v}} e_{ij}\otimes v_{ij}\in M_{n_v}\otimes L^{\infty}(\g)$ such that
\begin{equation}
\Delta(v_{ij})=\sum_{k=1}^{n_v}v_{ik}\otimes v_{kj}
\end{equation}
for all $1\leq i,j\leq n_v$. Furthermore, $v$ is called {\it irreducible} if 
\begin{equation}
\left \{T\in M_{n_v}: (T\otimes 1)v = v(T\otimes 1)\right\}=\left \{\lambda \cdot \text{Id}_{n_v}: \lambda\in \Comp\right\}.
\end{equation}

Two unitary representations $v,w\in M_n\otimes L^{\infty}(\g)$ are called {\it unitarily equivalent} if there exists a unitary matrix $U\in M_n$ such that
\begin{equation}
    (U\otimes 1)v(U^*\otimes 1)=w,
\end{equation}
and we write $v\cong w$ in this case. We denote by $\text{Irr}(\g)$ the set of all irreducible unitary representations up to the unitary equivalence. For each equivalence class $\alpha\in \text{Irr}(\g)$, let us take a representative element as
\begin{equation}
u^{\alpha}=\displaystyle \sum_{i,j=1}^{n_{\alpha}}e_{ij}\otimes u^{\alpha}_{ij}\in M_{n_{\alpha}}\otimes L^{\infty}(\g).    
\end{equation}
Here $n_{\alpha}$ is called the {\it classical dimension} of $u^{\alpha}$. An important distinction from the case of compact groups is that the {\it conjugate representation} 
\begin{equation}
(u^{\alpha})^c=\displaystyle \sum_{i,j=1}^{n_{\alpha}}e_{ij}\otimes \left (u^{\alpha}_{ij}\right )^*\in M_{n_{\alpha}}\otimes L^{\infty}(\g)   
\end{equation}
is not necessarily a unitary, but there exists an invertible positive matrix $Q_{\alpha}\in M_{n_{\alpha}}$ such that 
\begin{equation}
u^{\overline{\alpha}}=(Q_{\alpha}^{\frac{1}{2}}\otimes 1)(u^{\alpha})^c (Q_{\alpha}^{-\frac{1}{2}}\otimes 1)
\end{equation}
is a unitary representation. Here, we may assume that $Q_{\alpha}$ is diagonal with a suitable choice of an orthonormal basis. Moreover, the matrix $Q_{\alpha}$ is uniquely determined with the condition $\text{Tr}(Q_{\alpha})=\text{Tr}(Q_{\alpha}^{-1})$. We write
\begin{equation}
d_{\alpha}=\text{Tr}(Q_{\alpha})=\text{Tr}(Q_{\alpha}^{-1}),
\end{equation}
and call $d_{\alpha}$ the {\it quantum dimension} of $u^{\alpha}$. A compact quantum group $\g$ is called {\it Kac type} if one of the following equivalent conditions holds:
\begin{itemize}
    \item $Q_{\alpha}=\text{Id}_{n_{\alpha}}$ for all $\alpha\in \text{Irr}(\g)$.
    \item $d_{\alpha}=n_{\alpha}$ for all $\alpha\in \text{Irr}(\g)$.
    \item the Haar state $h$ is tracial, i.e. $h(ab) = h(ba)$ for all $a,b\in L^{\infty}(\g)$.
\end{itemize}

Let $\displaystyle v=\sum_{i,j=1}^{m}e_{ij}\otimes v_{ij}\in M_{m}\otimes L^{\infty}(\g)$ and $\displaystyle w=\sum_{s,t=1}^{n}e_{st}\otimes w_{st}\in M_{n}\otimes L^{\infty}(\g)$ be representations of $\g$. The {\it direct sum} is defined by
\begin{align}
v\oplus w & = \left [ \begin{array}{cccccc}
v_{11}&\cdots&v_{1m}&&&\\
\vdots&\ddots&\vdots&&&\\
v_{m1}&\cdots&v_{mm}&&&\\
&&&w_{11}&\cdots&w_{1n}\\
&&&\vdots&\ddots&\vdots\\
&&&w_{n1}&\cdots&w_{nn}
\end{array}\right ] 
\end{align}
as an element in $M_{m+n}\otimes L^{\infty}(\g)$, and the {\it tensor product} is defined by
\begin{align}
v\tp w & = \sum_{i,j=1}^{m}\sum_{s,t=1}^{n} e_{ij}\otimes e_{st}\otimes v_{ij}w_{st}
\end{align}
as an element in $M_{m}\otimes M_{n}\otimes L^{\infty}(\g)$. Both the direct sum $v\oplus w$ and the tensor product $v\tp w$ are representations of $\g$. For any finite-dimensional unitary representation $u$ of $\g$, there exist $\gamma_1$, $\gamma_2$, $\cdots$, $\gamma_m\in \text{Irr}(\g)$ such that 
\begin{equation}
u \cong u^{\gamma_1}\oplus u^{\gamma_2}\oplus \cdots \oplus u^{\gamma_m}.
\end{equation}
In this case, we call $\gamma_k$ a component of irreducible decomposition of $u$ and write $\gamma_k\subseteq u$.

\subsection{Schur's orthogonality and Fourier series}

The space of polynomials 
\begin{equation}
\text{Pol}(\g)=\text{span}\left \{u^{\alpha}_{ij}: \alpha\in \text{Irr}(\g),1\leq i,j\leq n_{\alpha}\right\}    
\end{equation}
is a weak$*$-dense $*$-subalgebra of $L^{\infty}(\g)$. The Schur's orthogonality relation states that
\begin{align}
\label{Schur1}
h\left ((u^{\alpha}_{ij})^* u^{\beta}_{st}\right )&=\frac{\delta_{\alpha,\beta}\delta_{i,s}\delta_{j,t}(Q_{\alpha})_{ii}^{-1}}{d_{\alpha}}\\
\label{Schur2}
h\left (u^{\alpha}_{ij}  (u^{\beta}_{st})^*\right )&=\frac{\delta_{\alpha,\beta}\delta_{i,s}\delta_{j,t}(Q_{\alpha})_{jj}}{d_{\alpha}}
\end{align}
for all $\alpha,\beta\in \text{Irr}(\g)$, $1\leq i,j\leq n_{\alpha}$ and $1\leq s,t\leq n_{\beta}$. We denote by $L^2(\g)$ the completion of $\text{Pol}(\g)$ with respect to the inner product
\begin{equation}
\la f,g\ra = h(g^*f)
\end{equation}
for all $f,g\in \text{Pol}(\g)$. Then \eqref{Schur1} implies that 
\begin{equation}
\left \{\sqrt{d_{\alpha}(Q_{\alpha})_{ii}}u^{\alpha}_{ij}: \alpha\in \text{Irr}(\g),1\leq i,j\leq n_{\alpha}\right\}    
\end{equation}
is an orthonormal basis of $L^2(\g)$. In particular, the set of all {\it characters} $\left\{\chi_{\alpha}\right\}_{\alpha\in \text{Irr}(\g)}$ is an orthonormal subset in $L^2(\g)$, where $\chi_{\alpha}=\sum_{i=1}^{n_{\alpha}}u^{\alpha}_{ii}$.

While $h(fg)\neq h(gf)$ on non-Kac $\g$ in general, the (non-trivial) {\it modular automorphism group} $\sigma=(\sigma_z)_{z\in \Comp}$ of the Haar state $h$ explains how those quantities are related. The automorphism $\sigma_z$ is determined by
\begin{equation}
\sigma_z(u^{\alpha}_{st})=(Q_{\alpha})_{ss}^{iz}(Q_{\alpha})_{tt}^{iz}u^{\alpha}_{st}
\end{equation}
for all $\alpha\in \text{Irr}(\g)$ and $1\leq s,t\leq n_{\alpha}$, and we have 
\begin{equation}
    h(fg)=h(g\sigma_{-i}(f))
\end{equation}
for all $f,g\in \text{Pol}(\g)$.

We define the {\it noncommutative $L^1$-space} on $\g$ as the predual $L^1(\g)=L^{\infty}(\g)_*$. The {\it Fourier coefficient} of a normal bounded linear functional $\omega\in L^{1}(\g)$ at $\alpha\in \text{Irr}(\g)$ is given by $\widehat{\omega}(\alpha)\in M_{n_{\alpha}}$ whose entries are
\begin{equation}
    \widehat{\omega}(\alpha)_{ij}=\omega \left ( (u^{\alpha}_{ji})^* \right )
\end{equation}
for all $1\leq i,j\leq n_{\alpha}$, and we call
\begin{equation}
    \omega\sim \sum_{\alpha\in \text{Irr}(\g)}\sum_{i,j=1}^{n_{\alpha}}d_{\alpha}(\widehat{\omega}(\alpha)Q_{\alpha})_{ij}u^{\alpha}_{ji}
\end{equation}
the {\it Fourier series} of $\omega\in L^1(\g)$. Note that we have a natural contractive embedding $\iota: L^{\infty}(\g)\hookrightarrow L^1(\g)$ given by 
\begin{equation}\label{eq10}
\iota(a)=\omega_a(\cdot)=h(\cdot a)
\end{equation}
for all $a\in L^{\infty}(\g)$. When we identify $f\in \text{Pol}(\g)\subseteq L^{\infty}(\g)$ with the associated normal bounded functional $\omega_f=h(\cdot f )\in L^1(\g)$, we obtain
\begin{equation}
    f=\sum_{\alpha\in \text{Irr}(\g)}\sum_{i,j=1}^{n_{\alpha}}d_{\alpha}(\widehat{f}(\alpha)Q_{\alpha})_{ij}u^{\alpha}_{ji},
\end{equation}
where the Fourier coefficient $\widehat{f}(\alpha)$ is defined as $\widehat{\omega_f}(\alpha)$.

\subsection{Noncommutative $L^p$-spaces and the crossed product}

For a general description, let $\M\subseteq B(H)$ be a von Neumann algebra with a normal faithful positive linear functional $\phi\in \M_*^+$. See \cite{Te81,PiXu03,HJX10} for more details of noncommutative $L^p$-spaces. Since $\M$ is naturally embedded into $\M_*$ similarly as in \eqref{eq10}, we have a compatible pair $(\M,\M_*)$ of Banach spaces, and the {\it noncommutative $L^p$-space} is defined as the complex interpolation space
\begin{equation}
    L^p(\M,\phi)=(\M,\M_*)_{\frac{1}{p}}
\end{equation}
for all $1\leq p\leq \infty$. 

Let $1\leq p_0,p_1\leq \infty$ and $\theta\in [0,1]$ such that $\displaystyle \frac{1}{p}=\frac{1-\theta}{p_0}+\frac{\theta}{p_1}$. Then 
\begin{equation}\label{ineq20}
\norm{f}_{L^p(\M,\phi)}\leq \norm{f}_{L^{p_0}(\M,\phi)}^{1-\theta}\norm{f}_{L^{p_1}(\M,\phi)}^{\theta}
\end{equation}
for all $f\in \M$.

Although the normal faithful positive functional $\phi\in \M_*^+$ is non-tracial in general, the elements in $L^p(\M,\phi)$ can be realized concretely in a larger noncommutative {\it tracial} measure space. Using the modular automorphism $\sigma^{\phi}=(\sigma_t^{\phi})_{t\in \mathbb{R}}$, the {\it crossed product} 
\begin{equation}
\mathcal{R}=\M\rtimes_{\sigma^{\phi}}\mathbb{R}\subseteq B(L^2(\re,H))    
\end{equation}
is defined as the von Neumann algebra generated by two families of operators, $\left\{\pi(x)\right\}_{x\in \M}$ and $ \left\{\lambda_s\right\}_{s\in \re}$ given by
\begin{align}
    (\pi(x)\xi)(t)&=\sigma_{-t}^{\phi}(x)\xi(t)\\
    (\lambda_s\xi)(t)&=\xi(t-s)
\end{align}
for all $x\in \M$ and $s,t\in \re$. Here, $\pi:\M\hookrightarrow \mathcal{R}$ is a normal faithful $*$-representation, so we can identify $\M$ with $\pi(\M)$.

By the Stone's theorem, there exists a densely defined injective positive operator $D=e^{A}:\mathcal{D}(A)\rightarrow L^2(\re,H)$ such that $\lambda(t)=D^{it}$ for all $t\in \re$, where $\mathcal{D}(A)$ is the set of all $\xi\in L^2(\re,H)$ such that 
\begin{equation}
A\xi=\lim_{s\rightarrow 0}\frac{1}{is}(\lambda_s\xi -\xi)   
\end{equation}
exists in $L^2(\re,H)$.

%The extended positive cone $\widetilde{\mathcal{M}}_+$ is the set of lower semi-continuous maps $m:\mathcal{M}_*^+\rightarrow [0,\infty]$ satisfying
%\begin{enumerate}
%    \item $m(\lambda \phi)=\lambda m(\phi)$ for all $\phi\in \mathcal{M}_*^+$ and $\lambda \geq 0$,
%    \item $m(\phi+\psi)=m(\phi)+m(\psi)$ for all $\phi, \psi \in \mathcal{M}_*^+$.
%\end{enumerate}

%The following map $\Phi:\widetilde{\mathcal{R}}_+\rightarrow \widetilde{\pi(L^{\infty}(\g))}_+$ given by
%\begin{equation}
%    \Phi(x)=\int_{\re}\widehat{\sigma}_s(x)ds,~x\in \mathcal{R}_+,
%\end{equation}
%is an operator-valued weight from $\mathcal{R}$ to $\pi(L^{\infty}(\g))$ in the sense of \cite{Ha79a,Ha79b}, i.e.
%\begin{enumerate}
%    \item $\Phi(\lambda x)=\lambda \Phi(x)$ for all $x\in \mathcal{R}_+$ and $\lambda\geq 0$,
%    \item $\Phi(x+y)=\Phi(x)+\Phi(y)$ for all $x,y\in \mathcal{R}_+$,
%    \item $\Phi(a^*xa)=a^*\Phi(x)a$ for all $x\in \mathcal{R}_+$ and $a\in \pi(L^{\infty}(\g))$.
%\end{enumerate}

%\begin{equation}
%    \widehat{\varphi}=\varphi\circ \pi^{-1}\circ \Phi.
%\end{equation}

For any compactly supported weak$*$-continuous function $f:\re\rightarrow \M$, let us consider the following operator
\begin{equation}
x_f=\int_{\re}\pi(f(s))\lambda_s ds.
\end{equation}
The set of all such operators $x_f$ forms a weak$*$-dense $*$-subalgebra of $\mathcal{R}$. 

For each normal semifinite weight $\varphi:\M_+\rightarrow [0,\infty]$, there exists the dual normal semifinite weight $\widehat{\varphi}:\mathcal{R}_+\rightarrow [0,\infty]$ determined by
\begin{equation}
  \widehat{\varphi}(x_f) = \widehat{\varphi}\left ( \int_{\re}\pi(f(s))\lambda_s ds \right )=\varphi(f(0))
\end{equation}
for any positive $x_f$. In addition, the dual weight $\widehat{\varphi}$ is faithful if the given weight $\varphi$ is faithful. 

Although $\widehat{\phi}$ of the given faithful $\phi\in \M_*^+$ is non-tracial in general, the following weight
\begin{equation}
\tau(\cdot)=\widehat{\phi}(D^{-\frac{1}{2}} \cdot D^{-\frac{1}{2}} )
\end{equation}
is the unique normal semifinite faithful tracial weight on $\mathcal{R}$ satisfying
\begin{equation}
   \sigma_t^{\widehat{\phi}}(x)=\lambda_t \sigma_t^{\tau}(x) \lambda_t^{-1}=D^{it}xD^{-it},~t\in \re.
\end{equation}

Let $L^0(\mathcal{R},\tau)$ be the space of all closed densely defined operators $T$ on $L^2(\re,H)$ satisfying the following conditions:
\begin{enumerate}
    \item $T$ is affiliated with $\mathcal{R}$, i.e. $TV=VT$ for all unitaries $V$ in the commutant algebra $\mathcal{R}'$.
    \item $T$ is $\tau$-measurable, i.e. for any $\delta>0$ there exists a projection $p\in \mathcal{R}$ such that $\text{ran}(p)\subseteq \mathcal{D}(T)$ and $\tau(1-p)\leq \delta$.
\end{enumerate}

For any normal positive functional $\omega\in \M_*^+$, there exists a unique positive $D_{\omega}\in L^0(\mathcal{R},\tau)$ such that
\begin{equation}
\widehat{\omega}(x)=\tau(D_{\omega}^{\frac{1}{2}} x D_{\omega}^{\frac{1}{2}} )=\tau(x D_{\omega})
\end{equation}
for all $x\in \mathcal{R}_+$, and the modular automorphism of $\widehat{\omega}$ is given by $\sigma^{\widehat{\omega}}_t(x)=D_{\omega}^{it}xD_{\omega}^{-it}$. In particular,  we have $D_{\phi}=D=e^A$ for the given normal faithful positive functional $\phi\in \M_*^+$. See Lemma 2.2 and Corollary 2.6 of \cite{Te81} for more details.

\subsection{Dual action and Haagerup $L^p$-spaces}

To define {\it the Haagerup $L^p$-spaces} $L^{p,H}(\M,\phi)$, let us introduce the dual action $\widehat{\sigma}^{\phi}$ of $\widehat{\re}\cong \re$ on $\mathcal{R}$ using a canonical identification of the dual group $\widehat{\re}$ with $\re$. Specifically, the dual action $\widehat{\sigma}^{\phi}$ is given by 
\begin{equation}
    \widehat{\sigma}^{\phi}_s(x)=w_sxw_s^*
\end{equation}
for all $s\in \re$ and $x\in \mathcal{R}$, where the unitary representation $w:\widehat{\re}\cong \re\rightarrow \mathcal{U}(L^2(\re,H))$ is given by
\begin{equation}
    (w_s\xi)(t)=e^{-ist}\xi(t)
\end{equation}
for all $\xi\in L^2(\re,H)$. In particular, we have
\begin{align}
\widehat{\sigma}^{\phi}_s(\pi(a)\lambda_t)=e^{-ist}\pi(a)\lambda_t
\end{align}
for all $a\in \M$ and $s,t\in \re$. The subalgebra $\pi(\M)$ is characterized as the fixed points for the dual action inside $\mathcal{R}$, i.e.
\begin{equation}
    \pi(\M)=\left\{ x\in \mathcal{R}:~\widehat{\sigma}^{\phi}_s(x)=x~\forall s\in \re \right\}.
\end{equation}

The Haagerup $L^p$-space is defined as
\begin{equation}
    L^{p,H}(\M,\phi)=\left\{x\in L^0(\mathcal{R},\tau): \widehat{\sigma}^{\phi}_s(x)=e^{-\frac{s}{p}}x~\forall s\in \re\right\}
\end{equation}
for any $1\leq p \leq \infty$. In particular, we have $L^{\infty,H}(\M,\phi)=\pi(\M)$ for $p=\infty$, and
\begin{equation}\label{eq21}
    L^{1,H}(\M,\phi)=\text{span}\left\{D_{\omega}\in L^0(\mathcal{R},\tau): \omega\in \M_*^+\right\}
\end{equation}
for $p=1$ where $D_{\omega}$ is the unique positive $\tau$-measurable operator associated to $\omega\in \M_*^+$ discussed in Subsection 2.3. Here, \eqref{eq21} is thanks to the following identities $\widehat{\omega}=\widehat{\omega}\circ \widehat{\sigma}^{\phi}_t$ and $e^{-t}\tau=\tau\circ \widehat{\sigma}_t^{\phi}$.

Note that $\text{Tr}(D_{\omega})=\omega(1)$ extends to a contractive positive normal linear functional on $L^{1,H}(\M,\psi)$, which we call {\it trace}. If $x\in L^{p,H}(\M,\phi)$ and $y\in L^{q,H}(\M,\phi)$ such that $\frac{1}{p}+\frac{1}{q}=1$ with $p,q\geq 1$, then $xy,yx\in L^{1,H}(\M,\phi)$ and we have
\begin{equation}\label{traciality}
   \text{Tr}\left ( xy \right ) = \text{Tr}\left ( yx\right ) .
\end{equation}

Now we can explain how the complex interpolation space $L^p(\M,\phi)$ is identified with the Haagerup $L^p$-space $L^{p,H}(\M,\phi)$. Let $x=u|x|$ be the polar decomposition of $x\in L^0(\mathcal{R},\tau)$. Then $x\in L^{p,H}(\M,\phi)$ is equivalent to that $u\in \M$ and $|x|^p\in L^{1,H}(\M,\phi)$. In this case, the $L^p$-norm of $x\in L^{p,H}(\M,\phi)$ is defined by
\begin{equation}
    \norm{x}_{p,H}=\text{Tr}(|x|^p)^{\frac{1}{p}}.
\end{equation}
By \cite[Section 9]{Kos84}, a linear map $\iota_p:\M\hookrightarrow L^{p,H}(\M,\phi)$ given by 
\begin{equation}
\iota_p(x)=\pi(x)D^{\frac{1}{p}}
\end{equation}
extends to an onto isometry $\iota_p:L^p(\M,\phi)\rightarrow L^{p,H}(\M,\phi)$ for all $1\leq p\leq \infty$.

As in the tracial setting, we have a natural H{\"o}lder inequality on the Haagerup $L^p$-spaces. For any $x\in L^{p,H}(\M,\phi)$ and $y\in L^{q,H}(\M,\phi)$ such that $\frac{1}{p}+\frac{1}{q}=\frac{1}{r}$ with $r\geq 1$, we have $xy\in L^{r,H}(\M,\phi)$ and
\begin{equation}\label{Holder}
\norm{xy}_{L^{r,H}(\M,\phi)}\leq \norm{x}_{L^{p,H}(\M,\phi)}\cdot \norm{y}_{L^{q,H}(\M,\phi)}.
\end{equation}

We define the noncommutative vector-valued $L^1$-space as $S^1_n[\M_*]=S^1_n\widehat{\otimes}\M_*\cong (M_n\otimes_{\text{min}} \M)_*$ where $\widehat{\otimes}$ and $\otimes_{\text{min}}$ are the projective and the injective tensor product in the category of operator spaces respectively. For the general cases $1\leq p\leq \infty$, the noncommutative vector-valued $L^p$-space is realized as the complex interpolation space
\begin{equation}
    S^p_n[L^p[\M,\phi]]=\left ( M_n\otimes_{\text{min}}\M, S^1_n\widehat{\otimes}\M_* \right )_{\frac{1}{p}}=L^p\left (\M_n,\phi_n\right )
\end{equation}
by \cite[Corollary 1.4]{Pi98b}, where $\M_n=M_n\overline{\otimes} \M$, $\phi_n=\text{tr}_n\otimes \phi$ and $\text{tr}_n$ is the usual trace on $M_n$. Thus, we obtain
\begin{equation}
   \norm{f}_{S^p_n[L^p[\M,\phi]]} = \left [\left (\text{tr}_n\otimes \text{Tr}\right )\left (\left |(\text{id}\otimes \pi)(f)(\text{Id}_n \otimes D^{\frac{1}{p}})\right |^p \right ) \right ]^{\frac{1}{p}}
\end{equation}
for any $f\in M_n\otimes_{\text{min}} \M$ and $1\leq p<\infty$. See \cite{Pi98b,Pi03} for more details of the noncommutative vector-valued $L^p$-spaces.

%for any $x,y\in \text{Pol}(\g)$, we have
%\begin{align}
%&\left [ \text{Tr}\left ( \left | xyD^{\frac{1}{r}}\right | \right ) \right ]^{\frac{1}{r}}=\norm{\pi(xy)}_{L^{r,H}(\g)}\\
%&\leq \norm{\pi(x)}_{L^{p,H}(\g)}\cdot \norm{\pi(y)}_{L^{q,H}(\g)}=\left [ \text{Tr}\left ( \left | xD^{\frac{1}{p}}\right | \right ) \right ]^{\frac{1}{p}}\cdot \left [ \text{Tr}\left ( \left | yD^{\frac{1}{q}}\right | \right ) \right ]^{\frac{1}{q}}.
%\end{align}

%A crucial advantage from this transference is that $\mathcal{R}$ is semisimple and there exists a unique normal semifinite faithful trace $\tau$ on $\mathcal{R}$.

%\begin{theorem}
%For any central Fourier series $f\displaystyle \sim \sum_{k=0}^{\infty} a_k\chi_k\in \text{Pol}(O_F^+)$ with $\norm{F}>1$ we %have
%\begin{equation}
%\norm{f}_{L^4(O_F^+)}\lesssim \norm{f}_{L^2(O_F^+)}.
%\end{equation}
%\end{theorem}

%Generalize the above!!

\section{A Khintchine inequality for central Fourier series}

We will focus on the so-called {\it central Fourier series with operator coefficients}, which is of the form
\begin{equation}
    f\sim \sum_{\alpha\in \text{Irr}(\g)}x_{\alpha}\otimes \chi_{\alpha}
\end{equation}
with an arbitrary sequence $(x_{\alpha})_{\alpha\in \text{Irr}(\g)}\subseteq M_n$. A key observation of this paper is the following smoothing effect of the automorphism $\sigma_{-\frac{i}{2p}}$ on the $L^{p}$-norm of the characters $\chi_{\alpha}$.

\begin{lemma}\label{lem-main}
For any $p=2^k$ with a natural number $k$, we have
\begin{equation}\label{eq41}
\norm{\sigma_{-\frac{i}{2p}}(\chi_{\alpha})}_{L^{p}(\g)}\leq \norm{\chi_{\alpha}}_{L^{\infty}(\g)}^{1-\frac{2}{p}}\left ( \frac{n_{\alpha}}{d_{\alpha}} \right )^{\frac{1}{p}}.
\end{equation}
\end{lemma}

\begin{proof}
For the first step $k=1$ ($\Leftrightarrow p=2$), we have 
\begin{align}
  \norm{\sigma_{-\frac{i}{4}}(\chi_{\alpha})}_{L^2(\g)}^2&=  \norm{\sum_{j=1}^{n_{\alpha}} (Q_{\alpha})_{jj}^{\frac{1}{2}}u^{\alpha}_{jj}}_{L^2(\g)}^2\\
  =&\sum_{j=1}^{n_{\alpha}}(Q_{\alpha})_{jj}\frac{(Q_{\alpha})_{jj}^{-1}}{d_{\alpha}}=\frac{n_{\alpha}}{d_{\alpha}}=\norm{\chi_{\alpha}}_{\infty}^{0}\cdot \left (\frac{n_{\alpha}}{d_{\alpha}}\right )^{1}.
\end{align}

Now let us assume that \eqref{eq41} holds for $p=2^k$. Then we have
\begin{align}
&\norm{\sigma_{-\frac{i}{4p}}(\chi_{\alpha})}_{L^{2p}(\g)}^{2p}=\norm{\pi\left (\sigma_{-\frac{i}{4p}}(\chi_{\alpha})\right )D^{\frac{1}{2p}}}_{L^{2p,H}(\g)}^{2p}\\
&=\text{Tr}\left (\left |D^{\frac{1}{4p}}\pi(\chi_{\alpha})D^{-\frac{1}{4p}}D^{\frac{1}{2p}}\right |^{2p}\right )=\text{Tr}\left (\left |D^{\frac{1}{4p}}\pi(\chi_{\alpha})D^{\frac{1}{4p}}\right |^{2p}\right )\\
&=\text{Tr}\left ( \left (D^{\frac{1}{4p}}\pi(\chi_{\alpha})^*D^{\frac{1}{2p}}\pi(\chi_{\alpha})D^{\frac{1}{4p}}\right )^{p}\right )\\
&= \text{Tr}\left (D^{\frac{1}{4p}} \left ( \pi(\chi_{\alpha})^*D^{\frac{1}{2p}}\pi(\chi_{\alpha})D^{\frac{1}{2p}} \right )^{p-1} \pi(\chi_{\alpha})^*D^{\frac{1}{2p}}\pi(\chi_{\alpha})D^{\frac{1}{4p}} \right ) \\
&=\text{Tr}\left ( \left (\pi(\chi_{\alpha})^*D^{\frac{1}{2p}}\pi(\chi_{\alpha})D^{\frac{1}{2p}} \right )^{p}\right ).
%&=\text{Tr}(\chi_{\alpha}^*(D^{\frac{1}{4p}}\chi_{\alpha}D^{\frac{1}{4p}})\cdot \chi_{\alpha}^*(D^{\frac{1}{4p}}\chi_{\alpha}D^{\frac{1}{4p}})\cdots \chi_{\alpha}^*(D^{\frac{1}{4p}}\chi_{\alpha}D^{\frac{1}{4p}})),
\end{align}
The last equality is thanks to the traciality \eqref{traciality}. Then, applying the noncommutative H{\" o}lder inequality \eqref{Holder}, we obtain
\begin{align}
&\norm{\sigma_{-\frac{i}{4p}}(\chi_{\alpha})}_{L^{2p}(\g)}^{2}=\left [ \text{Tr}\left ( \left (\pi(\chi_{\alpha})^*D^{\frac{1}{2p}}\pi(\chi_{\alpha})D^{\frac{1}{2p}} \right )^{p}\right ) \right ]^{\frac{1}{p}}\\
&\leq \norm{\pi(\chi_{\alpha})^*}_{L^{\infty,H}(\g)}\cdot \left [ \text{Tr}\left (\left |D^{\frac{1}{2p}}\pi(\chi_{\alpha})D^{\frac{1}{2p}}\right |^{p} \right ) \right ]^{\frac{1}{p}} \\
&= \norm{\pi(\chi_{\alpha}^*)}_{L^{\infty,H}(\g)} \cdot \left [\text{Tr}\left (\left |D^{\frac{1}{2p}}\pi(\chi_{\alpha})D^{-\frac{1}{2p}}D^{\frac{1}{p}}\right |^{p}\right )\right ]^{\frac{1}{p}}\\
&=\norm{\chi_{\alpha}}_{L^{\infty}(\g)}\cdot \norm{\pi\left (\sigma_{-\frac{i}{2p}}(\chi_{\alpha})\right )D^{\frac{1}{p}}}_{L^{p,H}(\g)}\\
&=\norm{\chi_{\alpha}}_{L^{\infty}(\g)}\cdot \norm{\sigma_{-\frac{i}{2p}}(\chi_{\alpha})}_{L^{p}(\g)}.
\end{align}
Thus, we can conclude that
\begin{align}
&\norm{\sigma_{-\frac{i}{4p}}(\chi_{\alpha})}_{L^{2p}(\g)}\leq \norm{\chi_{\alpha}}_{L^{\infty}(\g)}^{\frac{1}{2}}\cdot \norm{\sigma_{-\frac{i}{2p}}(\chi_{\alpha})}_{L^{p}(\g)}^{\frac{1}{2}}\\
&\leq \norm{\chi_{\alpha}}_{L^{\infty}(\g)}^{\frac{1}{2}}\cdot \norm{\chi_{\alpha}}_{L^{\infty}(\g)}^{\frac{1}{2}-\frac{1}{p}}\left ( \frac{n_{\alpha}}{d_{\alpha}} \right )^{\frac{1}{2p}}=\norm{\chi_{\alpha}}_{L^{\infty}(\g)}^{1-\frac{1}{p}}\left ( \frac{n_{\alpha}}{d_{\alpha}} \right )^{\frac{1}{2p}}
\end{align}
where the second inequality comes from the inductive assumption.
\end{proof}

Now, let us exhibit the main result, an operator-valued Khintchine inequality for an arbitrary central Fourier series with the following constant
\begin{equation}\label{coeff}
   K_p= \left (\sum_{\alpha\in \text{Irr}(\g)}  \norm{\chi_{\alpha}}_{L^{\infty}(\g)}^{2-\frac{4}{p}}  \left ( \frac{n_{\alpha}}{d_{\alpha}} \right )^{\frac{2}{p}} \right )^{\frac{1}{2}}\in [1,\infty].
\end{equation}

\begin{theorem}\label{thm-main}
Let $p = 2^k$ with a natural number $k$. Then
\begin{equation}
\norm{f}_{L^p(\g;S^p_n)}\leq K_p \cdot \norm{\left (\sum_{\alpha\in \text{Irr}(\g)} x_{\alpha}^*x_{\alpha} \right )^{\frac{1}{2}}}_{S^p_n}
\end{equation}
for any central Fourier series $f\displaystyle = \sum_{\alpha\in \text{Irr}(\g)} x_{\alpha}\otimes \chi_{\alpha}\in M_n\otimes \text{Pol}(\g)$ with operator coefficients.
\end{theorem}

\begin{proof}
The case $k=1$ ($\Leftrightarrow p=2$) is immediate, so let us suppose that $k\geq 2$. Recall that the vector-valued $L^p$-norm of $f=\displaystyle \sum_{\alpha\in \text{Irr}(\g)} x_{\alpha}\otimes  \chi_{\alpha}\in M_n\otimes \text{Pol}(\g)$ is given by 
\begin{align}
&\norm{f}_{L^p(\g;S^p_n)}=\left [\left (\text{tr}_n\otimes \text{Tr}\right )\left (\left |(\text{id}\otimes \pi)(f)(\text{Id}_n\otimes D^{\frac{1}{p}})  \right |^p \right ) \right ]^{\frac{1}{p}}
\end{align}
and note that
\begin{align}
    &\left |(\text{id}\otimes \pi)(f)(\text{Id}_n\otimes D^{\frac{1}{p}})\right |^p\\
    &=\left ( (\text{Id}_n\otimes D^{\frac{1}{p}})(\text{id}\otimes \pi)(f)^* (\text{id}\otimes \pi)(f) (\text{Id}_n\otimes D^{\frac{1}{p}}) \right )^{\frac{p}{2}}\\
    &=\left ( (\text{Id}_n\otimes D^{\frac{1}{p}})(\text{id}\otimes \pi)(f^*f)(\text{Id}_n\otimes D^{-\frac{1}{p}})(\text{Id}_n\otimes D^{\frac{2}{p}}) \right )^{\frac{p}{2}}\\
    &=\left ( \left (\text{id}\otimes \pi\right )\left ((\text{id}\otimes \sigma_{-\frac{i}{p}})(f^*f)\right )(\text{Id}_n\otimes D^{\frac{2}{p}}) \right )^{\frac{p}{2}}.
\end{align}
Thus, we have
\begin{align}
    \norm{f}_{L^p(\g;S^p_n)}^2=\norm{(\text{id}\otimes \sigma_{-\frac{i}{p}})(f^*f)}_{S^{\frac{p}{2}}_n[L^{\frac{p}{2}}(\g)]}
\end{align}

Since $\text{span}\left\{\chi_{\alpha}\right\}_{\alpha\in \text{Irr}(\g)}$ is a $*$-subalgebra of $L^{\infty}(\g)$, the element $f^*f=\displaystyle \sum_{\alpha_1,\alpha_2\in \text{Irr}(\g)} x_{\alpha_1}^*x_{\alpha_2}\otimes \chi_{\alpha_1}^*\chi_{\alpha_2}$ can be written as a linear combination of characters $\displaystyle \sum_{\beta\in \text{Irr}(\g)}y_{\beta}\otimes \chi_{\beta}$, where the operator coefficients are determined by 
\begin{align}
&y_{\beta}=(\text{id}\otimes h)\left ((\text{Id}_n\otimes \chi_{\beta}^*) \cdot f^*f\right )\\
&=\sum_{\alpha_1\in \text{Irr}(\g)} x_{\alpha_1}^*   \sum_{\alpha_2\in \text{Irr}(\g) } x_{\alpha_2}\cdot h( \chi_{\beta}^* \chi_{\alpha_1}^*\chi_{\alpha_2})\\
&=\sum_{\alpha_1\in \text{Irr}(\g)} x_{\alpha_1}^*   \sum_{\alpha_2\in \text{Irr}(\g) } x_{\alpha_2}\cdot h(\chi_{\alpha_1}\chi_{\beta}\chi_{\alpha_2}^*) = T^* \left ( \text{Id}_n\otimes \chi_{\beta}\right ) T
\end{align}
where $T=\sum_{\gamma\in \text{Irr}(\g) } x_{\gamma}\otimes |\chi_{\gamma}^*\ra$ with a linear map $|\chi_{\gamma}^*\ra: \Comp\rightarrow L^2(\g)$ mapping $z$ to $z\cdot  \chi_{\gamma}^*$. Then we have
\begin{align}
  \label{eq31}  \norm{y_{\beta}}_{S^{\frac{p}{2}}_n}&\leq \norm{T^*}_{S^p}\cdot \norm{\text{Id}_n\otimes \chi_{\beta}}_{M_n\overline{\otimes }L^{\infty}(\g)}\cdot \norm{T}_{S^p}\\
  \label{eq32}  &=\norm{\chi_{\beta}}_{L^{\infty}(\g)}\cdot \norm{\left (\sum_{\gamma\in \text{Irr}(\g)}x_{\gamma}^*x_{\gamma}\right )^{\frac{1}{2}}}_{S^p_n}^2.
\end{align}
Combining the above observations with the triangle inequality and Lemma \ref{lem-main}, we reach the following conclusion
\begin{align}
   & \norm{(\text{id}\otimes \sigma_{-\frac{i}{p}})(f^*f)}_{S^{\frac{p}{2}}_n[L^{\frac{p}{2}}(\g)]}\\
   &\leq \sum_{\beta\in \text{Irr}(\g)}\norm{y_{\beta}\otimes \sigma_{-\frac{i}{p}}(\chi_{\beta})}_{S^{\frac{p}{2}}_n[L^{\frac{p}{2}}(\g)]}\\
   &\leq \sum_{\beta\in \text{Irr}(\g)} \norm{\chi_{\beta}}_{L^{\infty}(\g)} \norm{\left (\sum_{\gamma\in \text{Irr}(\g)}x_{\gamma}^*x_{\gamma}\right )^{\frac{1}{2}}}_{S^p_n}^2 \cdot \norm{\chi_{\beta}}_{L^{\infty}(\g)}^{1-\frac{4}{p}}\left ( \frac{n_{\beta}}{d_{\beta}}\right )^{\frac{2}{p}}.
\end{align}

\end{proof}

\begin{remark}
Theorem \ref{thm-main} can be rephrased as that $E=\text{Irr}(\g)$ is a completely bounded $\Lambda(p)$-set under the assumption $K_p<\infty$ in a natural sense of \cite{Har99} within the framework of compact quantum groups.
\end{remark}

Then, some standard complex interpolation arguments allow us to show that all $L^p$-norms of the central Fourier series are equivalent in the scalar-valued case.

\begin{corollary}\label{cor-main}
Suppose that $K_p<\infty$ for infinitely many $p=2^k$ with natural numbers $k$. Then, for any $1\leq r,s<\infty$, we have
\begin{equation}
    \norm{f}_{L^{r}(\g)}\approx \norm{f}_{L^{s}(\g)}
\end{equation}
for arbitrary central Fourier series $f = \sum_{\alpha\in \text{Irr}(\g)}a_{\alpha}\chi_{\alpha}\in \text{Pol}(\g)$.
\end{corollary}
\begin{proof}

It is enough to prove that, for any $1\leq r<\infty$, the following inequality
\begin{equation}
   \norm{f}_{L^r(\g)}\lesssim \norm{f}_{L^1(\g)}
\end{equation}
holds for all central Fourier series $f= \sum_{\alpha\in \text{Irr}(\g)}a_{\alpha} \chi_{\alpha}\in \text{Pol}(\g)$.

First of all, suppose that $K_p<\infty$ for $p=2^k$ with a natural number $k\geq 2$. Then there exists $\theta\in (0,1)$ such that $\frac{1}{2}=\frac{1-\theta}{1}+\frac{\theta}{p}$ ($\Leftrightarrow \theta=\frac{p}{2(p-1)}$), and we have
\begin{align}
\norm{f}_{L^2(\g)} & \leq \norm{f}_{L^1(\g)}^{1-\theta}\norm{f}_{L^p(\g)}^{\theta}\\
&\leq \norm{f}_{L^1(\g)}^{1-\theta}\cdot K_p^{\theta} \norm{f}_{L^2(\g)}^{\theta}
\end{align}
by \eqref{ineq20} and Theorem \ref{thm-main}. This implies 
\begin{equation}\label{eq33}
\norm{f}_{L^2(\g)}\leq K_p^{\frac{\theta}{1-\theta}}\norm{f}_{L^1(\g)}=K_p^{\frac{p}{p-2}}\norm{f}_{L^1(\g)}.
\end{equation}
Combining \eqref{eq33} with Theorem \ref{thm-main}, we can conclude that
\begin{equation}\label{eq30}
    \norm{f}_{L^p(\g)}\leq K_p\norm{f}_{L^2(\g)}\leq K_p^{\frac{2p-2}{p-2}}\norm{f}_{L^1(\g)}.
\end{equation}

For the general cases $1\leq r<\infty$, there exists $p=2^k$ with a natural number $k\geq 2$ such that $r<2^k$ and $K_p<\infty$. Let us write $\frac{1}{r}=\frac{1-\theta'}{1}+\frac{\theta'}{p}$ ($\Leftrightarrow \theta'=\frac{p(r-1)}{r(p-1)}$) with $\theta'\in [0,1)$. Then, by \eqref{ineq20} and \eqref{eq30}, we obtain
\begin{equation}
    \norm{f}_{L^r(\g)}\leq \norm{f}_{L^1(\g)}^{1-\theta'}\norm{f}_{L^p(\g)}^{\theta'}=K_p^{\frac{2p(r-1)}{r(p-2)}}\norm{f}_{L^1(\g)}.
\end{equation}

\end{proof}

\section{Main examples}

In this Section, we demonstrate that the Khintchine inequality with operator coefficients holds for arbitrary central Fourier series in a large class of non-Kac compact quantum groups. More precisely, we will prove that
\begin{align}
    \label{eq44} K_p^2=\sum_{\alpha\in \text{Irr}(\g)}  \norm{\chi_{\alpha}}_{L^{\infty}(\g)}^{2-\frac{4}{p}}  \left ( \frac{n_{\alpha}}{d_{\alpha}} \right )^{\frac{2}{p}}<\infty
\end{align}
holds for all $p>2$ when $\g$ is one of the following compact quantum groups
\begin{itemize}
    \item Drinfeld-Jimbo $q$-deformations $G_q$ with $0<q<1$,
    \item Non-Kac free orthogonal quantum groups $O_F^+$,
    \item Non-Kac quantum automorphism group $\g_{\text{aut}}(B,\psi)$ with a $\delta$-form.
\end{itemize}

Our unified strategy to cover all the aforementioned examples is as follows. In Subsections 4.1, 4.2 and 4.3, we will prove that there exists $r\in (0,1)$ such that 
\begin{equation}
\frac{n_{\alpha}}{d_{\alpha}}\lesssim r^{|\alpha|}    
\end{equation}
where $|\cdot|:\text{Irr}(\g)\rightarrow \n$ is a natural length function for each $\g$ in the above list. Assuming the existence of such $r\in (0,1)$, we can obtain
\begin{align}
    K_p^2=&\sum_{k=0}^{\infty}\sum_{\alpha\in \text{Irr}(\g): |\alpha|=k}\norm{\chi_{\alpha}}_{L^{\infty}(\g)}^{2-\frac{4}{p}}  \left ( \frac{n_{\alpha}}{d_{\alpha}} \right )^{\frac{2}{p}}\\
    & \lesssim \sum_{k=0}^{\infty} \left ( \sum_{\alpha\in \text{Irr}(\g): |\alpha|=k}\norm{\chi_{\alpha}}_{L^{\infty}(\g)}^{2}  \right ) r^{\frac{2k}{p}},
\end{align}
so it is enough for us to see that $\displaystyle \sum_{\alpha\in \text{Irr}(\g): |\alpha|=k}\norm{\chi_{\alpha}}_{L^{\infty}(\g)}^{2} $
grows polynomially when $k$ increases.

(1) For the Drinfeld-Jimbo $q$-deformations $\g=G_q$, we have  
\begin{equation}
   \sum_{\alpha\in \text{Irr}(\g): |\alpha|=k}\norm{\chi_{\alpha}}_{L^{\infty}(\g)}^{2}= \sum_{\alpha\in \text{Irr}(\g): |\alpha|=k}n_{\alpha}^2,
\end{equation}
thanks to coamenability of $G_q$, and the right-hand side grows polynomially for $k$ as in the classical situation by \cite[Theorem 2.4.7 (3)]{NeTu13} and \cite{BaVe09}. 

(2) For the other cases $\g=O_F^+$ (resp. $\g=\g_{\text{aut}}(B,\psi)$), let us consider $G=SU(2)$ (resp. $G=SO(3)$). Then $\text{Irr}(\g)\cong \text{Irr}(G)\cong \n_0$ with the same fusion rules by \cite{Ba96} (resp. \cite{Br13}). For each character $\chi_{k}\in \text{Pol}(\g)$, let us denote by $\chi_{k}'\in \text{Pol}(G)$ the corresponding character. Then we have
\begin{equation}
   \norm{\chi_{k}}_{L^{\infty}(\g)}=\norm{\chi_{k}'}_{L^{\infty}(G)}=k+1~(\text{resp. }2k+1),
\end{equation}
by \cite[Proposition 6.7]{Wa17} and \cite{Ba96} (resp. \cite{Ba99b}), so we can conclude that
\begin{align}
    \sum_{\alpha\in \text{Irr}(\g): |\alpha|=k}\norm{\chi_{\alpha}}_{L^{\infty}(\g)}^{2}=\norm{\chi_k}_{L^{\infty}(\g)}^2\lesssim (k+1)^2 .
\end{align}

Hence, the only remaining thing to prove is the existence of $r\in (0,1)$ satisfying $\displaystyle \frac{n_{\alpha}}{d_{\alpha}}\lesssim r^{|\alpha|}$, which will be discussed in the following Subsections 4.1, 4.2, and 4.3 for each class of examples.

\subsection{Drinfeld-Jimbo $q$-deformations}

Let $G$ be a simply connected compact semisimple Lie group and let $\mathfrak{g}$ be the associated complexified Lie algebra with a Cartan subalgebra $\mathfrak{h}$. The {\it adjoint map} $\text{ad}:\mathfrak{g}\rightarrow \text{End}(\mathfrak{g})$ is given by
\begin{equation}
    \text{ad}_X(Y)=[X,Y]
\end{equation}
for all $X,Y\in \mathfrak{g}$. A non-zero linear functional $\alpha\in \mathfrak{h}^*$ is called a {\it root} if there exists non-zero $X\in \mathfrak{g}$ such that
\begin{equation}
    \text{ad}_H(X)=[H,X]=\alpha(H)X
\end{equation}
for all $H\in \mathfrak{h}$. The space of all roots $\alpha$ is denoted by $\Phi\subseteq \mathfrak{h}^*\setminus \left\{0\right\}$, and the associated non-trivial eigenspaces
\begin{equation}
    \mathfrak{g}_{\alpha}=\left\{X\in \mathfrak{g}: \forall H\in \mathfrak{h},~\text{ad}_H(X)=\alpha(H)X\right\}
\end{equation}
are called {\it root spaces}. Furthermore, there exists a subset $\Phi^+$ of $\Phi$ satisfying the following:
\begin{itemize}
    \item The subset $\Phi^+$ contains exactly one of $\alpha$ and $-\alpha$ for each $\alpha\in \Phi$.
    \item If $\alpha,\beta\in \Phi^+$ $(\alpha\neq \beta)$ satisfies $\alpha+\beta\in \Phi$, then $\alpha+\beta\in \Phi^+$.
\end{itemize}
All elements in $\Phi^+$ are called {\it positive roots}, and a positive root is called {\it simple} if it cannot be written as the sum of two positive roots. Let us write the simple roots as $\alpha_1,\alpha_2,\cdots, \alpha_r$. Then every positive root $\alpha\in \Phi^+$ is a linear combination of simple roots with non-negative integral coefficients.

There exists a non-degenerate symmetric $\text {ad}$-invariant form $(\cdot,\cdot)$ on $\mathfrak{g}$ whose restriction to the real Lie algebra is negative definite. Note that the non-degenerate symmetric form $(\cdot,\cdot)$ provides a natural identification $\mathfrak{h}^*\cong \mathfrak{h}$, so it induces a natural non-degenerate symmetric form on $\mathfrak{h}^*$. Furthermore, the form $(\cdot,\cdot)$ restricts to an inner product on $\mathfrak{h}_0^*=\text{span}_{\mathbb{R}}\Phi$. Another important basis of $\mathfrak{h}_0^*$ consists of the so-called {\it fundamental weights} $\omega_1,\omega_2,\cdots,\omega_r\in \mathfrak{h}_0^*$ determined by the following relations 
\begin{equation}
    \frac{2(\omega_i,\alpha_j)}{(\alpha_j,\alpha_j)}=\delta_{ij}
\end{equation}
for all $1\leq i,j\leq r$. The set of all $\z$-linear combinations of the fundamental weights is called the {\it weight lattice} for $\mathfrak{g}$, which we denote by $P$. The subset of all $\n_0$-linear combinations of the fundamental weights is denoted by $P_+$.

Now, let $0<q<1$ and consider the Drinfeld-Jimbo $q$-deformation $G_q$. Then $\text{Irr}(G_q)$ is identified with $P_+$. See \cite[Section 10.1]{ChPr95} and \cite[Theorem 2.4.7]{NeTu13}) for more details. Let us denote by $\rho=\sum_{i=1}^r \omega_i$, which is equal to $\frac{1}{2}\sum_{\alpha\in \Phi^+}\alpha$ (\cite[Proposition 2.69]{Kn02} and \cite{NeTu13}). The following proposition is important to establish the Khintchine inequality with operator coefficients for arbitrary central Fourier series on $G_q$.

\begin{proposition}
Let $t_i=q^{(\omega_i,2\rho)}$ for all $1\leq i\leq r$. Then
\begin{equation}
    \norm{Q_{\mu}}_{\infty} = t_1^{-\mu_1}t_2^{-\mu_2}\cdots t_r^{-\mu_r}
\end{equation}
for all $\mu=\sum_{i=1}^r \mu_i \omega_i \in P_+\cong \text{Irr}(G_q)$. In particular, $\frac{n_{\mu}}{d_{\mu}}$ decreases exponentially as $\sum_{i=1}^r \mu_i\rightarrow \infty$.
\end{proposition}
\begin{proof}
For the associated irreducible representation $u^{\mu}$ of $G_q$, the eigenvalues of the modular matrix $Q_{\mu}$ are of the form $q^{(\nu,-2\rho)}$ where $\nu$ is a weight for $u^{\mu}$ (\cite[Proposition 2.4.10]{NeTu13} and \cite{VoYu20}). Let us prove that $(\nu,2\rho)$ is maximized when $\nu$ is the highest weight $\mu$. Indeed, since $\mu$ is the highest weight for $u^{\mu}$, the difference $\mu-\nu$ can be written as $\sum_{i=1}^r m_i \alpha_i$ with $m_i\geq 0$ and we have
\begin{align}
   & (\mu,2\rho)-(\nu,2\rho)=\sum_{i,j=1}^r 2m_i (\alpha_i, \omega_j)=\sum_{i=1}^r m_i(\alpha_i,\alpha_i) \geq 0.
\end{align}
Thus, we can conclude that
\begin{align}
&\norm{Q_{\mu}}_{\infty}= q^{-(\mu,2\rho)}=q^{-\sum_{i=1}^r \mu_i  (\omega_i,2\rho)}=\prod_{i=1}^r \left ( q^{(\omega_i,2\rho )} \right )^{-\mu_i}.
\end{align}

For the last conclusion, since $n_{\mu}$ is polynomially growing by \cite[Theorem 2.4.7 (3)]{NeTu13} and \cite[Theorem 2.1]{BaVe09}, it is enough to see that $d_{\mu}$ is exponentially growing as $\sum_{i=1}^r \mu_i\rightarrow \infty$. Note that  
\begin{equation}
t_i=q^{(\omega_i,2\rho)}=q^{\sum_{\alpha\in \Phi^+} (\omega_i,\alpha)}\leq q^{(\omega_i,\alpha_i)}=q^{\frac{(\alpha_i,\alpha_i)}{2}}<1    
\end{equation}
for all $1\leq i\leq r$, implying $\displaystyle t=\max_{1\leq i\leq r}t_i\in (0,1)$. Thus, we obtain the following exponentially growing lower bound:
\begin{equation}
d_{\mu}\geq \norm{Q_{\mu}}_{\infty}\geq t^{-(\mu_1+\mu_2+\cdots+\mu_r)}.
\end{equation}

\end{proof}

%Then the space of irreducible representations $\text{Irr}(G_q)$ is identified with $\n_0^r$ for some natural number $r$. Let us denote by $u^{(1)}$, $u^{(2)}$, $\cdots$, $u^{(r)}$ the irreducible unitary representations associated to the standard generators $e_1$, $e_2$, $\cdots$, $e_r$ of $\n_0^r$ respectively. Then for any ${\bf n}=(n_1,n_2,\cdots,n_r)\in \n_0^r$ the irreducible unitary representation $u^{\bf n}=u^{(n_1,n_2,\cdots,n_r)}$ is the highest weight module of $(u^{(1)})^{\tiny\tp \normalsize n_1}\tp (u^{(2)})^{\tiny\tp \normalsize n_2}\tp \cdots \tp (u^{(r)})^{\tiny\tp \normalsize n_r}$. The natural $\ell^1$-length function $|{\bf n}|=n_1+n_2+\cdots+n_r$ can be understood as
%\begin{equation}
%\min\left \{k\in \n_0: ~u^{\bf n}\text{ is a subrepresentation of }\left (u^{(1)}\oplus \cdots \oplus u^{(r)}\right )^{\tiny \tp \normalsize k} \right \}.
%\end{equation}
%See \cite{BaVe09,ChPr95,NeTu13,Bo82,Ha15} for more details. In particular, for each ${\bf n}=(n_1,n_2,\cdots,n_r)\in \n_0^r$, \cite[Proposition 4.10]{NeTu13} states that
%\begin{equation}
%\norm{Q_{\bf n}}_{\infty}\geq q_1^{-n_1}q_2^{-n_2}\cdots q_r^{-n_r}
%\end{equation}
%for some positive numbers $q_1,q_2,\cdots, q_r>0$ (which are determined by the maximal torus and the associated of root system of $G$).

\subsection{Free orthogonal quantum groups}

Let $O_F^+$ be a non-Kac free orthogonal quantum group with $F\in M_N$ such that $N_q=\text{Tr}(F^*F)>N$. See \cite{Wa95,Ba96,Ti08} for more details. Since Subsection 4.1 covers the case $N=2$, let us focus on the other general cases $N\geq 3$. Recall that the space $\text{Irr}(O_F^+)$ is identified with $\n_0$ with the following fusion rule \cite{Ba96}:
\begin{equation}
 u^{(k)}\tp u^{(1)}\cong u^{(k-1)}\oplus u^{(k+1)},~k\in \n.
\end{equation}

Both the quantum dimension $d_k$ and the classical dimension $n_k$ are described by the {\it Chebyshev polynomials} $(f_k)_{k\in \n_0}$ of the second kind. The polynomial $f_k$ is explicitly given by 
\begin{equation}\label{eq-chebyshev}
    f_k(t)=
    \left ( \frac{1}{2} \right )^{k+1}\frac{(t+\sqrt{t^2-4})^{k+1}-(t-\sqrt{t^2-4})^{k+1}}{\sqrt{t^2-4}}  
\end{equation}
for all $t>2$, and we have $d_k=f_k(N_q)$ and $n_k=f_k(N)$. Combining these facts with the following identity
\begin{equation}\label{eq-limit}
    \lim_{k\rightarrow \infty}f_k(t)\cdot \left ( \frac{t+\sqrt{t^2-4}}{2} \right )^{-(k+1)}=\frac{1}{\sqrt{t^2-4}},
\end{equation}
we obtain the following estimate
\begin{align}
    \frac{n_k}{d_k}\approx \left ( \frac{N+\sqrt{N^2-4}}{N_q+\sqrt{N_q^2-4}} \right )^{k+1}.
\end{align}
Since $N<N_q$ is equivalent to $ \frac{N+\sqrt{N^2-4}}{N_q+\sqrt{N_q^2-4}}<1$, we can conclude that $ \frac{n_k}{d_k}$ decreases exponentially as $k\rightarrow \infty$.

\subsection{Quantum automorphism group}

Let $\g$ be a non-Kac quantum automorphism group $\g_{\text{aut}}(B,\psi)$. We refer the readers to \cite{Wa98,Ba99b,Ba02,Br13} for more details of $\g_{\text{aut}}(B,\psi)$ with a $\delta$-form $\psi$. By \cite{Ba99b,Br13}, the space $\text{Irr}(\g_{\text{aut}}(B,\psi))$ is identified with $\n_0$ with the following fusion rule:
\begin{equation}
u^{(k)}\tp u^{(1)}\cong u^{(k-1)}\oplus u^{(k)}\oplus u^{(k+1)},~k\in \n.
\end{equation}

To compute the quantum dimension and the classical dimension, let us consider the following polynomials $g_k$ such that $g_k(4)=2k+1$ and $g_k(x)=f_{2k}(\sqrt{x})$ for all $x>4$. Here, $f_{2k}$ is the $2k$-th Chebyshev polynomial given in \eqref{eq-chebyshev}, and the identity \eqref{eq-limit} implies
\begin{equation}\label{eq40}
    \lim_{k\rightarrow \infty}g_k(x)\cdot \left ( \frac{x-2+\sqrt{x(x-4)}}{2} \right )^{-k}=\frac{\sqrt{x}+\sqrt{x-4}}{2\sqrt{x-4}}
\end{equation}
for all $x>4$. Note that $d_k=g_k(d_1)=f_{2k}(\sqrt{d_1})$ and $n_k=g_k(n_1)=f_{2k}(\sqrt{n_1})$ respectively, by \cite[Theorem 3.8]{Br13}.

Firstly, if $\text{dim}(B)=4$, then $n_k=g_k(4)=2k+1$ grows polynomially and the above identity \eqref{eq40} implies 
\begin{equation}
d_k\approx  \left ( \frac{d_1-2+\sqrt{d_1(d_1-4)}}{2} \right )^k.
\end{equation}
Then, the non-Kac condition $d_1>n_1=4$ implies $\frac{d_1-2+\sqrt{d_1(d_1-4)}}{2}>1$, so the quantum dimension $d_k$ has an exponential growth.

For the general cases $\text{dim}(B)\geq 5$, we have
\begin{align}
    \frac{n_k}{d_k} \approx \left ( \frac{n_1-2+\sqrt{n_1(n_1-4)}}{d_1-2+\sqrt{d_1(d_1-4)}} \right )^{k}
\end{align}
up to constants by \eqref{eq40}, and the non-Kac condition, i.e., $n_1<d_1$, implies that the right-hand side decreases exponentially.

\newpage 

\emph{Acknowledgements}: The author thanks Professor {\'E}ric Ricard for his insightful comments, particularly regarding the consideration of operator coefficients. The author was supported by the National Research Foundation of Korea (NRF) grant funded by the Ministry of Science and ICT (MSIT) (No. 2020R1C1C1A01009681), and also by Samsung Science and Technology Foundation under Project Number SSTF-BA2002-01.

\bibliographystyle{alpha}
\bibliography{Youn23}

\end{document}